\documentclass[10pt, a4paper, reqno]{amsart}

\usepackage[utf8]{inputenc}

\usepackage[english]{babel}

\usepackage{amsmath}
\usepackage{amsfonts}
\usepackage{amssymb}
\usepackage{amsthm}
\usepackage{hyperref}
\usepackage{graphicx}

\usepackage{enumerate}

\usepackage[numbers]{natbib}

\nocite{*}

\usepackage[left=3cm, right=3cm, top=3.5cm, bottom=4cm]{geometry}

\theoremstyle{plain}
\newtheorem{theorem}{Theorem}[section]
\newtheorem{corollary} {Corollary}[section]

\theoremstyle{definition}
\newtheorem{remark}{Remark}[section]

\numberwithin{equation}{section}

\title{On the distribution of the maximum of the telegraph process}

\author{F. Cinque and E. Orsingher$^1$}
\address{$^1$Department of Statistical Sciences, Sapienza University of Rome, Italy.}
\email{cinque.1700526@studenti.uniroma1.it, enzo.orsingher@uniroma1.it}
\date{\scriptsize \texttt{\today}}

\begin{document}

\begin{abstract}
In this paper we present the distribution of the maximum of the telegraph process in the cases where the initial velocity is positive or negative with an even and an odd number of velocity reversals. For the telegraph process with positive initial velocity a reflection principle is proved to be valid while in the case of an initial leftward displacement the conditional distributions are perturbed by a positive probability of never visiting the half positive axis.
\\Various relationships are established among the mentioned four classes of conditional distributions of the maximum.
\\The unconditional distributions of the maximum of the telegraph process are obtained for positive and negative initial steps as well as their limiting behaviour. Furthermore the cumulative distributions and the general moments of the conditional maximum are presented.
\end{abstract}
\maketitle


		\section{Introduction}

The telegraph process has been investigated by several researchers starting from the beginning of the Fifties of the past century.
The asymmetric telegraph process was studied by V.Cane (1975) and its explicit distribution obtained in Beghin et al. (2001) with two different approaches, one based on relativistic transformations, the other one by using the Fourier transforms of the governing equation.
\\In this paper we present a general picture of the distribution of the maximum of the homogeneous telegraph process, denoted throughout by $\mathcal{T}(t)$, $t\ge 0$, with a different initial speed $V(0)=\pm c$, $c>0$.
\\
The problem of finding the distribution
\begin{equation}\label{max1}
 P\lbrace \max_{0\le s\le t} \mathcal{T}(s) \in d\beta\ |\ V(0) = c, N(t) = 2k+1\rbrace
\end{equation}
for $k\in \mathbb{N}_0,\ 0<\beta<ct$, was firstly undertaken in Orsingher (1990), for some specific values of $k$. The obtained partial results inspired the formulation of the conjecture that for $0 < \beta < ct , \ k\ge 0$
\begin{equation}\label{congettura}
P\lbrace \max_{0\le s\le t} \mathcal{T}(s) \in d\beta\ |\ V(0) = c, N(t) = 2k+1\rbrace = \frac{B(k)(c^2t^2 - \beta^2)^k}{(ct)^{2k+1}}d\beta
\end{equation}
where $B(k)$ is the normalising constant.
\\The conjecture was proved by Foong (1992) and Foong and Kanno (1994). The starting point for proving (\ref{congettura}) was the derivation of the first-passage time from a Darling-Siegert's relationship by means of the Laplace transform.
\\Other works on this topic are due to Stadje and Zacks (2004) and Zacks (2004).
Our approach is based on the direct derivation of conditional distributions like (\ref{max1}) by means of the induction principle exploiting markovianity and homogeneity with respect to the Poisson times related to the reversals of velocities.
\\In this way we show that for $k\ge 0,\ 0<\beta<ct$
\begin{equation}\label{leggemax+d}
P\lbrace \max_{0\le s\le t} \mathcal{T}(s) \in d\beta\ |\ V(0) = c, N(t) = 2k+1\rbrace = 2\frac{(2k+1)!}{k!^2} \frac{(c^2t^2 - \beta^2)^k}{(2ct)^{2k+1}}d\beta
\end{equation}
and
\begin{equation}\label{ripmax+d}
P\lbrace \max_{0\le s\le t} \mathcal{T}(s) < \beta\ |\ V(0) = c, N(t) = 2k+1\rbrace = \frac{\beta}{ct} \sum_{j = 0}^k \binom{2j}{j}\Bigl(\frac{\sqrt{c^2t^2-\beta^2}}{2ct} \Bigr)^{2j}
\end{equation}
From (\ref{leggemax+d}) we also have that
\begin{equation}
P\lbrace \max_{0\le s\le t} \mathcal{T}(s) \in d\beta, N(t) \ odd\ |\ V(0) = c\rbrace = e^{-\lambda t}\frac{\lambda}{c}I_0\Bigl(\frac{\lambda}{c}\sqrt{c^2t^2-\beta^2}\Bigr)d\beta
\end{equation}
where $N(t)$, $t\ge 0$, is the number of Poisson events occured up to time $t$.
\\

Since
\begin{equation}\label{intromax+p}
P\lbrace \max_{0\le s\le t} \mathcal{T}(s) \in d\beta\ |\ V(0) = c, N(t) = 2k+2\rbrace = 2\frac{(2k+1)!}{k!^2} \frac{(c^2t^2 - \beta^2)^k}{(2ct)^{2k+1}}d\beta
\end{equation}
we have that
\begin{equation}
P\lbrace \max_{0\le s\le t} \mathcal{T}(s) \in d\beta, N(t) \ even \ |\ V(0) = c\rbrace =\frac{ e^{-\lambda t}}{c}\frac{\partial}{\partial t}I_0\Bigl(\frac{\lambda}{c}\sqrt{c^2t^2-\beta^2}\Bigr)d\beta
\end{equation}
so that
$$P \lbrace \max_{0\le s\le t} \mathcal{T}(s)  \in d \beta \ |\ V(0) = c\rbrace  =  $$
\begin{equation}\label{max+intro}
 = \frac{e^{-\lambda t}}{c} \Biggl[\lambda I_0\Bigl( \frac{\lambda }{c} \sqrt{c^2t^2 -\beta^2} \Bigr) + \frac{\partial }{\partial t}  I_0\Bigl( \frac{\lambda }{c} \sqrt{c^2t^2 -\beta^2} \Bigr) \Biggr]d\beta \ =\ 2P \lbrace  \mathcal{T}(t)  \in d \beta\rbrace
\end{equation} 
\\
Formula (\ref{max+intro}) shows that the reflection principle holds in the conditional case (\ref{leggemax+d}) and in the unconditional case (\ref{leggemax+d}) provided that $V(0) = c>0$.
\\

The situation is more complicated when $V(0) = -c$ and the explicit distributions are again obtained by induction by using result (\ref{leggemax+d}).
We show that
\begin{equation}\label{introlegge-p}
P\lbrace \max_{0\le s\le t} \mathcal{T}(s) \in d\beta\ |\ V(0) = -c, N(t) = 2k\rbrace = \frac{2(2k)!}{k!(k-1)!} \frac{(c^2t^2 - \beta^2)^{k-1}(ct-\beta)}{(2ct)^{2k}} d\beta
\end{equation}
for $0<\beta<ct$, $k\ge 1$ and
\begin{equation}\label{introsingolarita-p}
  P\lbrace \max_{0\le s\le t} \mathcal{T}(s) = 0\ |\ V(0) = -c, N(t) = 2k\rbrace = \binom{2k}{k}\frac{1}{2^{2k}}
\end{equation}
for $k\ge 0$.
\\
From (\ref{introlegge-p}) and (\ref{introsingolarita-p}) we derive the cumulative distribution function of the maximum as
$$P\lbrace \max_{0\le s\le t} \mathcal{T}(s) < \beta\ |\ V(0) = -c, N(t) = 2k\rbrace = $$
\begin{equation}\label{introrip-p}
 =\frac{\beta}{ct} \sum_{j=0}^{k-1} \Bigl( \frac{c^2t^2-\beta^2}{c^2t^2}\Bigr)^j \frac{1}{2^{2j}}\binom{2j}{j} +  \frac{(c^2t^2-\beta^2)^k }{(2ct)^{2k}}\binom{2k}{k}
\end{equation}
$$ = P\lbrace \max_{0\le s\le t} \mathcal{T}(s) < \beta\ |\ V(0) = c, N(t) = 2k-1\rbrace  +  \frac{(c^2t^2-\beta^2)^k }{(2ct)^{2k}}\binom{2k}{k}$$
\\
The last case examined shows that, for $0<\beta <ct$, the following relationships hold
$$P\lbrace \max_{0\le s\le t} \mathcal{T}(s) \in d\beta\ |\ V(0) = -c, N(t) = 2k+1\rbrace =$$
$$  =\Biggl( \frac{(2k+1)!}{(k-1)!(k+1)!} \frac{(c^2t^2 - \beta^2)^{k-1}(ct-\beta)}{(2ct)^{2k}} +\frac{(2k+1)!}{k!(k+1)!} \frac{(c^2t^2 - \beta^2)^{k}}{(2ct)^{2k+1}}  \Biggr)d\beta =  $$
$$=\frac{2k+1}{2k+2} P\lbrace \max_{0\le s\le t} \mathcal{T}(s) \in d\beta\ |\ V(0) = -c, N(t) = 2k\rbrace\ +$$
\begin{equation}\label{introsommapesata}
+\ \frac{1}{2k+2}P\lbrace \max_{0\le s\le t} \mathcal{T}(s) \in d\beta\ |\ V(0) = c, N(t) = 2k+1\rbrace
\end{equation}
\\
Thus the distribution is an average of distributions (\ref{introlegge-p}) and (\ref{leggemax+d}) with a prevailing weight of the first one.
\\
Finally
$$ P\lbrace \max_{0\le s\le t} \mathcal{T}(s) = 0\ |\ V(0) = -c, N(t) = 2k+1\rbrace = \binom{2k+1}{k}\frac{1}{2^{2k+1}} =$$
\begin{equation}\label{sing2}
=\frac{2k+1}{2k+2} P\lbrace \max_{0\le s\le t} \mathcal{T}(s) = 0\ |\ V(0) = -c, N(t) = 2k\rbrace
\end{equation}

\bigskip

\section{Preliminaries about the telegraph process}

The symmetric telegraph process has the form
\begin{equation}\label{ptdefint}
\mathcal{T}(t) = \int_0^t V(0)(-1)^{N(s)}ds 
\end{equation}
where $V(0)$ is a two valued symmetric random variable taking values $\pm c$ and $N(t)$ is a homogeneous Poisson process with rate $\lambda > 0$ and independent of $V(0)$. An alternative form of (\ref{ptdefint}) is

\begin{equation} \label{ptdef}
\mathcal{T}(t)=  V(0)\sum_{k=1}^{N(t)+1} (T_{k} - T_{k-1})(-1)^{k-1}
\end{equation}
with $0 = T_0 < T_1<...< T_{N(t)}<T_{N(t)+1} = t$. The random times $T_1,...T_{N(t)}$ are the arrival times of the Poisson process.
\\For $N(t) = n$, the instants $T_1, ...,T_n$ are uniformly distributed in the simplex and thus have density
\begin{equation}\label{unifsimplesso}
f(t_1, ..., t_n) = \frac{n!}{t^n}
\end{equation}
for $0<t_1<...<t_n<t$.
\\Because of the exchangeability of the random variables $T_k-T_{k-1}$ we can write the displacement (\ref{ptdef}), for $N(t) = n$, as
\begin{equation}\label{ptsviluppo}
\mathcal{T}_n(t) = V(0) \Bigl(T^+_n - T^-_n\Bigr)  =V(0) \Bigl( 2T_n^+ -t\Bigr)
\end{equation}
where $T^+_n$ is the time spent by the particle moving with the same direction of $V(0)$ and $T_n^-$ represents the time spent moving in the other direction. Clearly the total time is $t = T_n^+ + T_n^-$ and the variable $T^+_n$ is given by the sum of the $n^+$ displacements with speed equal to $V(0)$.
\\
The relationship (\ref{ptsviluppo}) permits us to write that
\begin{equation}\label{ptsviluppo2}
P\lbrace \mathcal{T}(t) < x | N(t) = n, V(0) = c\rbrace = \Bigg \{\begin{array}{l l} 
	\int_0^{\frac{x+ct}{2c}} f_{T^+_n}(z) dz & \ if\ V(0) = c \\
\\
	\int^t_{\frac{-x+ct}{2c}} f_{T^+_n}(z) dz & \ if\ V(0) =- c \\
			\end{array}
		\end{equation}
where  $f_{T^+_n}$ is the density of the time spent moving with speed $V(0)$. This density, as we show later, is equal to the density of the $n^+$-th order statistic of $n$ independent uniformly distributed, in $(0,t)$, random variables $Y_1, ... Y_n$ which reads
\begin{equation}\label{denstatord}
f_{Y_{(n^+)}}(z) = \frac{n}{t} \binom{n-1}{n^+-1}\Bigl(\frac{z}{t}\Bigr)^{n^+-1} \Bigl(1-\frac{z}{t} \Bigr)^{n-n^+}
\end{equation}
for $0<z<t$.
\\

In view of (\ref{ptsviluppo2}) and (\ref{denstatord}), by considering that $N(t) = 2k+1$, then we have that $n^+ = k+1$ and therefore
$$ P \lbrace \mathcal{T}(t) \in dx | N(t) = 2k+1, V(0) = c \rbrace =$$
\begin{equation}\label{pt+d} 
=\frac{d}{dx} \int_0^{\frac{x+ct}{2c}} \frac{2k+1}{t} \binom{2k}{k}\Bigl(\frac{z}{t}\Bigr)^k \Bigl(1-\frac{z}{t} \Bigr)^k dz = \frac{(2k+1)!}{k!^2} \frac{(c^2t^2-x^2)^k}{(2ct)^{2k+1}}dx
\end{equation}
for $|x|<ct$. 
\\Analogously we obtain that
\begin{equation}\label{pt-d}  P \lbrace \mathcal{T}(t) \in dx | N(t) = 2k+1, V(0) =- c \rbrace =\frac{(2k+1)!}{k!^2} \frac{(c^2t^2-x^2)^k}{(2ct)^{2k+1}}dx
\end{equation}
for $|x|<ct$. Distributions (\ref{pt+d}) and (\ref{pt-d}) are identical since in both cases there is the same number of rightward and leftward displacements. By means of them we also obtain that
\begin{equation}\label{ptd}  P \lbrace \mathcal{T}(t) \in dx | N(t) = 2k+1\rbrace =\frac{(2k+1)!}{k!^2} \frac{(c^2t^2-x^2)^k}{(2ct)^{2k+1}}dx
\end{equation}
for $|x|<ct$ and it coincides with formula (2.17) of De Gregorio et al. (2004).
\\

For an even number of changes of directions we have that (again $n^+ = k+1$), for $|x|<ct$
$$ P \lbrace \mathcal{T}(t) \in dx | N(t) = 2k, V(0) = c \rbrace =$$
\begin{equation}\label{pt+p} 
=\frac{d}{dx} \int_0^{\frac{x+ct}{2c}} \frac{2k}{t} \binom{2k-1}{k}\Bigl(\frac{z}{t}\Bigr)^k \Bigl(1-\frac{z}{t} \Bigr)^{k-1} dz = \frac{(2k)!}{k!(k-1)!} \frac{(ct+x)^k(ct-x)^{k-1}}{(2ct)^{2k}}dx
\end{equation}
and
\begin{equation}\label{pt-p}  P \lbrace \mathcal{T}(t) \in dx | N(t) = 2k, V(0) =- c \rbrace =\frac{(2k)!}{k!(k-1)!} \frac{(ct+x)^{k-1}(ct-x)^k}{(2ct)^{2k}}dx
\end{equation}
which lead to
\begin{equation}\label{ptp}  P \lbrace \mathcal{T}(t) \in dx | N(t) = 2k\rbrace =\frac{(2k)!}{k!(k-1)!} \frac{ct(c^2t^2-x^2)^{k-1}}{(2ct)^{2k}}dx
\end{equation}
that coincides with formula (2.18) of De Gregorio et al. (2004).
The reader can also notice that (\ref{ptp}), with $N(t) = 2k+2$, is equal to (\ref{ptd}).
\\

Finally, the unconditional probability law of the process is given by
\begin{equation}\label{ptsingolarita}
P\lbrace \mathcal{T}(t) = ct \rbrace = P\lbrace \mathcal{T}(t) = -ct \rbrace = \frac{e^{-\lambda t}}{2}
\end{equation}			
and
\begin{equation}\label{ptden}
P\lbrace \mathcal{T}(t) \in dx \rbrace = \frac{e^{- \lambda t}}{2c} \Biggl[ \lambda I_{0}\biggl(\frac{\lambda}{c}\sqrt{c^{2}t^{2}-x^{2}}\biggr) + \frac{\partial}{\partial t} I_{0} \biggl(\frac{\lambda}{c}\sqrt{c^{2}t^{2}-x^{2}}\biggr) \Biggr] dx
\end{equation}
for $|x|<ct$.
\\

Let $N(t) = n$. We now show that the increments $T_k-T_{k-1},\ 1\le k\le n$, appearing in (\ref{ptdef}) and composing $T_n^+$, have the same distribution of the increments of successive order statistics $Y_{(k)}-Y_{(k-1)}$ from uniformly distributed, in $(0,t)$, random variables $Y_1, ..., Y_n$.

First of all we recall that the two-fold joint distribution of order statistics $Y_{(k)}, Y_{(l)}, \ l>k$ from $n$ independent and identically distributed random variables, with cumulative distribution function $F$ and density $f$, is
$$P\lbrace Y_{(k)}\in dy, Y_{(l)}\in dz \rbrace = $$
\begin{equation}\label{congiuntastatord}
=\frac{n!}{(k-1)!(l-k-1)!(n-l-k)!}F^{k-1}(y)f(y)dy \Bigl(F(z)-F(y) \Bigr)^{l-k-1}f(z) dz \Bigl(1-F(z) \Bigr)^{n-l}
\end{equation}
for $0<y<z$.
\\

For uniformly distributed random variables in $(0,t)$, formula (\ref{congiuntastatord}) becomes
$$P\lbrace Y_{(k)}\in dy, Y_{(l)}\in dz \rbrace = $$
\begin{equation}
=\frac{n!}{(k-1)!(l-k-1)!(n-l-k)!}\Bigl(\frac{y}{t}\Bigr)^{k-1}\frac{dy}{t} \Bigl(\frac{z-y}{t} \Bigr)^{l-k-1}\frac{ dz}{t} \Bigl(\frac{t-z}{t} \Bigr)^{n-l}
\end{equation}
for $0<y<z<t$.
\\

The increment $Y_{(l)}-Y_{(k)}$ has distribution
\begin{equation}
P\lbrace Y_{(l)} - Y_{(k)}<w \rbrace = 1-\int_w^t \int_0^{z-w}P\lbrace Y_{(k)}\in dy, Y_{(l)}\in dz \rbrace
\end{equation}
and therefore, for $0<w<t$
$$P\lbrace Y_{(l)} - Y_{(k)}\in dw\rbrace = \frac{n!\ dw}{(k-1)!(l-k-1)!(n-l)!}\int_w^t (z-w)^{k-1}w^{l-k-1}(t-z)^{n-l}\frac{dz}{t^n} = $$
considering the change of variable $(z-w) = (t-w)y$
$$ = \frac{n!\ w^{l-k-1}\ dw}{(k-1)!(l-k-1)!(n-l)! t^n} \int_0^1 y^{k-1} (1-y)^{n-l}dz(t-w)^{n-l+k}  = $$
\begin{equation}\label{diffstatord}
= \frac{n! }{(l-k-1)!(n-l+k)!}\frac{w^{l-k-1}(t-w)^{n-l+k}}{t^n} dw
\end{equation}
and it only depends on the distance $l-k$. If $l = k+1$ (\ref{diffstatord}) reduces to 
\begin{equation}\label{diffstatord2}
P\lbrace Y_{(k+1)} - Y_{(k)}\in dw\rbrace = \frac{n}{t}\Bigl(1-\frac{w}{t}\Bigr)^{n-1} dw
\end{equation}
and is independent of $k$.

From the joint distribution (\ref{unifsimplesso}) of the instants of occurance $T_1, ..., T_n$ of a homogeneous Poisson process we extract the bivariate density
$$P\lbrace T_k \in dt_k, T_l \in dt_l |N(t) = n\rbrace = $$
$$ = \frac{n!}{t^n} \int_0^{t_k} dt_1 \cdot \cdot \cdot \int_{t_k-2}^{t_k} dt_{k-1} \cdot dt_k \cdot \int_{t_k}^{t_l} dt_{k+1} \cdot \cdot \cdot \int_{t_l-2}^{t_l} dt_{l-1} \cdot dt_l \cdot \int_{t_l}^t dt_{l+1} \cdot \cdot \cdot \int_{t_{n-1}}^t dt_n = $$
\begin{equation}\label{doppiapoisson}
= \frac{n!}{t^n}\frac{t_k^{k-1}}{(k-1)!}\frac{(t_l-t_k)^{l-k-1}}{(l-k-1)!}\frac{(t-t_l)^{n-l}}{(n-l)!}dt_k dt_l
\end{equation}
for $0<t_k<t_l<t$. 
\\This coincides with the joint distribution of the order statistics $Y_{(k)},\ Y_{(l)}
$ from a sample of $n$ independent uniformly distributed in $(0,t)$ random variables. Thus the distribution of $T_l-T_k$ coincides with that of $Y_{(l)} -Y_{(k)}$. Therefore, the one-step displacements $T_k-T_{k-1}$ have the same distribution of $Y_{(k)} -Y_{(k-1)}$.
Considering that the random time $T_n^+$ in (\ref{ptsviluppo}) is given by the sum of $n^+$ one-step displacements of the type $T_k-T_{k-1}$ and that we can express the $n^+$-th order statistics as
$$Y_{(n^+)} = Y_{(1)} + (Y_{(2)}-Y_{(1)}) + ... + (Y_{(n^+)}-Y_{(n^+-1)})$$
that it is a sum of $n^+$ increments of the type $Y_{(k)} -Y_{(k-1)}$, we have that $T_{n^+} \stackrel{d}{=} Y_{(n^+)}$.


\section{Maximum of the telegraph process with initial rightward velocity}

We begin by considering the case where $V(0) = c$ and the number of changes of direction is odd, that is $N(t) = 2k+1$, $k \in \mathbb{N}_0$. In this case the sample paths of the telegraph process
\begin{equation} \label{pt+dispari}
\mathcal{T}_{2k+1}(t) = cT_1-c(T_2-T_1)+...+c(T_{2k+1}-T_{2k})-c(t-T_{2k+1})
\end{equation}
consists of $k+1$ upward (or rightward) displacements and $k+1$ downward (or leftward) displacements. The relative maxima of (\ref{pt+dispari}) can be attained by its truncated displacements, truncation being performed at odd-order Poisson times.
\\The displacements being considered are therefore
$$\mathcal{T}(T_{2j+1}) =  cT_1-c(T_2-T_1)+...+c(T_{2j+1}-T_{2j})$$
for $1\le j\le k$, with $j+1$ rightward displacements and $j+1$ leftward ones.
\\

In this section we evaluate the following probability
$$P \lbrace \max_{0\le s\le t} \mathcal{T}(s) < \beta \ |\ N(t) = 2k+1, V(0) = c \rbrace  = $$
\begin{equation}\label{iniziomax+}
 = P\bigl\lbrace \cap_{j=0}^k \lbrace\mathcal{T}(T_{2j+1}) < \beta \rbrace \ |\  N(t) = 2k+1, V(0) = c\bigr\rbrace
\end{equation}
For $k = 0,1,2$ the distribution (\ref{iniziomax+}) was established by direct calculation. Therefore, for a rightward starting motion with one, three and five changes of direction  the distribution of the maxima showed a regularity leading to conjecture result (\ref{leggemax+d}).
In the next theorem we prove by induction that (\ref{leggemax+d}) holds.

\bigskip

\begin{theorem}\label{teoremamax+d}
Let $\lbrace \mathcal{T}(t) \rbrace_{t\ge0}$ a symmetric telegraph process, then
\begin{equation}\label{leggemax+d2}
P\lbrace \max_{0\le s\le t} \mathcal{T}(s) \in d\beta\ |\ V(0) = c, N(t) = 2k+1\rbrace \ =\ 2\frac{(2k+1)!}{k!^2} \frac{(c^2t^2 - \beta^2)^k}{(2ct)^{2k+1}}d\beta
\end{equation}
for $k\in \mathbb{N}_0,\ 0<\beta <ct$.
\end{theorem}

\bigskip

\begin{proof}
We establish a general recurrence relationship between the distribution of 
\\$(\max_{0\le s\le t} \mathcal{T}(s) | N(t) = n, V(0) = c)$ and $(\max_{0\le s\le t} \mathcal{T}(s) | N(t) = n-2, V(0) = c)$, for $n\ge2$.
\\
We must consider three cases:

Case 1. where $cT_1 <\beta$ and the maximum is obtained during the residual interval of time $(T_1, t)$. Thus, we need that in $T_2$ there is still enough time to reach $\beta$, meaning that $c(t-T_2)\ge\beta-2cT_1 + cT_2$, where $2cT_1-cT_2$ represents the position at time $T_2$ ;

Case 2. where the maximum $\beta$ is reached at time $T_1 = \frac{\beta}{c}$ and at time $T_2$ we still have time to overpass $\beta$, thus $c(t-T_2)\ge\beta-2c\frac{\beta}{c} + cT_2$;

Case 3. where the maximum $\beta$ is reached at time $T_1$ and at time $T_2$ we moved too far leftward that we surely never arrive at $\beta$ in the residual time interval $(T_2, t)$, thus  $c(t-T_2)<-\beta + cT_2$.
\\

In light of these three cases, the recurrence relationship writes
\begin{equation}\label{ricorrenza+}
P\lbrace \max_{0\le s\le t} \mathcal{T}(s) \in d\beta\ |\ V(0) = c, N(t) = n\rbrace  = 
\end{equation}
$$ = \int_0^{\frac{\beta}{c}} \int_{t_1}^{\frac{ct-\beta}{2c}+t_1}P\lbrace \max_{0\le s\le t-t_2} \mathcal{T}(s) \in d\beta\ |\ V(0) = c, N(t-t_2) = n-2, \mathcal{T}(0) = 2ct_1-ct_2\rbrace \cdot$$ 
$$\hspace{4cm}\cdot P\lbrace T_1 \in dt_1, T_2 \in dt_2\ | \ N(t) = n\rbrace + $$
$$ + \int_{\frac{\beta}{c}}^{\frac{ct+\beta}{2c}} P\lbrace \max_{0\le s\le t-t_2} \mathcal{T}(s) < \beta\ |\ V(0) = c, N(t-t_2) = n-2, \mathcal{T}(0) = 2\beta-ct_2\rbrace \cdot$$
$$\cdot P\lbrace T_1 \in \frac{d\beta}{c}, T_2 \in dt_2\ | \ N(t) = n\rbrace\ +\ \int_{\frac{ct+\beta}{2c}}^{t} P\lbrace T_1 \in \frac{d\beta}{c}, T_2 \in dt_2\ | \ N(t) =n\rbrace $$
where each term refers to each of the above scenarios.
\\

To prove the theorem we consider relation (\ref{ricorrenza+}) with $n = 2k+1$, $k\in \mathbb{N}_0$.
\\
In view of (\ref{leggemax+d}) (now induction hypothesis) and (\ref{doppiapoisson}), the first addend of (\ref{ricorrenza+}) becomes
$$ \int_0^{\frac{\beta}{c}} \int_{t_1}^{\frac{ct-\beta}{2c}+t_1}P\lbrace \max_{0\le s\le t-t_2} \mathcal{T}(s) \in d\beta\ |\ V(0) = c, N(t-t_2) = 2k-1, \mathcal{T}(0) = 2ct_1-ct_2\rbrace \cdot$$ 
$$\hspace{4cm}\cdot P\lbrace T_1 \in dt_1, T_2 \in dt_2\ | \ N(t) = 2k+1\rbrace =$$
$$ = d\beta\int_0^{\frac{\beta}{c}} dt_1 \int_{t_1}^{\frac{ct-\beta}{2c}+t_1} dt_2 \frac{2(2k-1)!}{(k-1)!^2\bigl(2c(t-t_2)\bigr)^{2k-1}} \Bigl(c^2(t-t_2)^2 -(\beta-2ct_1+ct_2)^2\Bigr)^{k-1} \cdot$$
$$\hspace{6cm}\cdot \frac{(2k+1)!(t-t_2)^{2k-1}}{(2k-1)!\ t^{2k+1}} = $$ 
(by considering the change of variable $w = t_2 -t_1$ in the inner integral)
$$ = \frac{d\beta\ 2(2k+1)!}{(k-1)!^2(2ct)^{2k-1}t^2} \int_0^{\frac{\beta}{c}} dt_1 \int_0^{\frac{ct-\beta}{2c}}  \Bigl(c^2(t-t_1-w)^2 -(\beta-ct_1+cw)^2\Bigr)^{k-1}dw = $$
$$ =\frac{d\beta\ 2(2k+1)!\ c^{2k-2}}{(k-1)!^2(2ct)^{2k-1}t^2} \int_0^{\frac{\beta}{c}} dt_1 \int_0^{\frac{ct-\beta}{2c}}  \Bigl((t+\frac{\beta}{c}-2t_1)(t-\frac{\beta}{c})-2w(t+\frac{\beta}{c}-2t_1) \Bigr)^{k-1}dw =$$
$$ =\frac{ d\beta\ (2k+1)!\ c^{k-2}}{k!(k-1)!(2ct)^{2k-1}t^2}  \int_0^{\frac{\beta}{c}} \bigl(t+\frac{\beta}{c}-2t_1\bigr)^{k-1}dt_1 (ct-\beta)^k = $$
\begin{equation}\label{addendo1+}
= 2\frac{(2k+1)!}{k!^2} \frac{(c^2t^2 - \beta^2)^k}{(2ct)^{2k+1}}d\beta - 2\frac{(2k+1)!}{k!^2} \frac{(ct- \beta)^{2k}}{(2ct)^{2k+1}}d\beta
\end{equation}
\\

The second term of (\ref{ricorrenza+}) is (remember that $n = 2k+1$)
$$\int_{\frac{\beta}{c}}^{\frac{ct+\beta}{2c}} P\lbrace \max_{0\le s\le t-t_2} \mathcal{T}(s) < -\beta +ct_2\ |\ V(0) = c, N(t-t_2) = 2k-1\rbrace \cdot$$
$$ \hspace{4cm} \cdot  P\lbrace T_1 \in \frac{d\beta}{c}, T_2 \in dt_2\ | \ N(t) = 2k+1\rbrace = $$
$$ = \int_{\frac{\beta}{c}}^{\frac{ct+\beta}{2c}}dt_2 \int_0^{ct_2-\beta} dz \frac{2(2k-1)!}{(k-1)!^2\bigl(2c(t-t_2)\bigr)^{2k-1}} \Bigl(c^2(t-t_2)^2 -z^2\Bigr)^{k-1} \cdot$$
$$\hspace{6cm}\cdot \frac{(2k+1)!(t-t_2)^{2k-1}}{(2k-1)!\ t^{2k+1}}\frac{d\beta}{c} = $$
$$ =  \frac{d\beta \ 2(2k+1)!}{(k-1)!^2c^{2k}t^{2k+1}2^{2k-1}} \int_0^{\frac{ct-\beta}{2c}} dz \int_{\frac{z+\beta}{c}}^{\frac{ct+\beta}{2c}} \Bigl(c^2(t-t_2)^2 -z^2\Bigr)^{k-1} dt_2 = $$
(by considering the change of variable $t_2 - \frac{z+\beta}{c} = w$ in the inner integral)
$$ =  \frac{d\beta \ 2(2k+1)!}{(k-1)!^2c^{2k}t^{2k+1}2^{2k-1}} \int_0^{\frac{ct-\beta}{2c}} dz \int_0^{\frac{ct-\beta}{2c}-\frac{z}{c}} \Bigl((ct-w-\beta)^2 -2z\bigl(c(t-w)-\beta\bigr)\Bigr)^{k-1} dw = $$
$$ =  \frac{d\beta \ 2(2k+1)!}{(k-1)!^2c^{2k}t^{2k+1}2^{2k-1}}  \int_0^{\frac{ct-\beta}{2c}} dw \Bigl( c(t-w)-\beta\Bigr)^{k-1} \int_0^{\frac{ct-\beta}{c}-cw} \Bigl( c(t-w)-\beta-z\Bigr)^{k-1}dz = $$
\begin{equation}\label{addendo2intermedio}
 =\frac{d\beta \ (2k+1)!}{k!(k-1)!c^{2k}t^{2k+1}2^{2k-1}} \Bigl[ \int_0^{\frac{ct-\beta}{2c}} \Bigl( ct-\beta-cw\Bigr)^{2k-1} dw - \int_0^{\frac{ct-\beta}{2c}} (cu)^{k}\Bigl( ct-\beta-cw\Bigr)^{k-1} dw \Bigr]
\end{equation}
The second integral in (\ref{addendo2intermedio}) can be reduced to a Beta integral
$$\int_0^{\frac{ct-\beta}{2c}} (cu)^{k}\Bigl( ct-\beta-cw\Bigr)^{k-1} dw = \hspace{2cm}\Bigl( cw = (ct-\beta)z\Bigr) = $$
\begin{equation}\label{secondointegrale} 
=\frac{(ct-\beta)^{2k}}{c} \int_0^{\frac{1}{2}}z^k(1-z)^{k-1}dz =  \frac{(ct-\beta)^{2k}}{c}\Bigl[ \frac{(k-1)!^2}{2^2(2k-1)} -\frac{1}{k2^{2k+1}}\Bigr]
\end{equation}
where the last equality can be obtained by developing the integral by parts.
\\
\\
In view of (\ref{secondointegrale}), (\ref{addendo2intermedio}) becomes
$$ \frac{d\beta \ 2(2k+1)!}{(k-1)!^2c^{2k}t^{2k+1}2^{2k-1}} \Bigl[\frac{(ct-\beta)^{2k}}{2ck}-\frac{(ct-\beta)^{2k}}{2^{2h}2ck} -\frac{(ct-\beta)^{2k}}{c}\Bigl( \frac{(k-1)!^2}{2^2(2k-1)} -\frac{1}{k2^{2k+1}}\Bigr) \Bigr] = $$
\begin{equation}\label{addendo2+}= 2\frac{(2k+1)!}{k!^2} \frac{(ct- \beta)^{2k}}{(2ct)^{2k+1}}d\beta - 2(2k+1)!\frac{(ct- \beta)^{2k}}{(2ct)^{2k+1}}d\beta
\end{equation}
\\

The third term of (\ref{ricorrenza+}) yields
$$\int_{\frac{ct+\beta}{2c}}^{t} P\lbrace T_1 \in \frac{d\beta}{c}, T_2 \in dt_2\ | \ N(t) =2k+1\rbrace = $$
\begin{equation}\label{addendo3+}
=\frac{d\beta}{c}  \int_{\frac{ct+\beta}{2c}}^{t} \frac{(2k+1)!(t-t_2)^{2k-1}}{(2k-1)!\ t^{2k+1}}dt_2=(2k+1)!\frac{(ct- \beta)^{2k}}{2^{2k}(ct)^{2k+1}}d\beta
\end{equation}
\\

By summing up (\ref{addendo1+}), (\ref{addendo2+}) and (\ref{addendo3+}) we obtain that (\ref{ricorrenza+}) corresponds to thesis (\ref{leggemax+d2}) and this concludes the proof.
\end{proof}

\bigskip

\begin{remark}
The result of Theorem (\ref{teoremamax+d}) shows that a reflection principle is valid for the telegraph process with $V(0) = c$. At the instant $T_1$ of first change of direction it is possible to construct a sample path "reflected" around the level $\mathcal{T}(T_1)$. If the original trajectory overshot level $\beta$ and at time $t$ was below $\beta$, then the new one obtained by glueing together the original step with the reflected displacement will be at time $t$ over $\beta$. Thus,

$$P \lbrace \max_{0\le s\le t} \mathcal{T}(s)  > \beta , \mathcal{T}(t) < \beta \ |\ N(t)  = 2k+1, V(0) = c\rbrace  =  $$
$$ = P \lbrace \max_{0\le s\le t} \mathcal{T}(s)  > \beta , \mathcal{T}(t) > \beta \ |\ N(t)  = 2k+1, V(0) = c\rbrace  = $$ 
$$ = P \lbrace \mathcal{T}(t) > \beta \ |\ N(t)  = 2k+1, V(0) = c\rbrace $$ \hfill $\Diamond$
\end{remark}

\bigskip

\begin{theorem}\label{teoremamax+p}
Let $\lbrace \mathcal{T}(t) \rbrace_{t\ge0}$ a symmetric telegraph process, then
$$P\lbrace \max_{0\le s\le t} \mathcal{T}(s) \in d\beta\ |\ V(0) = c, N(t) = 2k+2\rbrace =$$
\begin{equation}\label{leggemax+p2}
 = P\lbrace \max_{0\le s\le t} \mathcal{T}(s) \in d\beta\ |\ V(0) = c, N(t) = 2k+1\rbrace \ = \ 2\frac{(2k+1)!}{k!^2} \frac{(c^2t^2 - \beta^2)^k}{(2ct)^{2k+1}}d\beta
\end{equation}
for $k\in \mathbb{N},\ 0<\beta <ct$ and
\begin{equation}\label{max+0}
P\lbrace \max_{0\le s\le t} \mathcal{T}(s) =ct\ |\ V(0) = c, N(t) = 0\rbrace = 1
\end{equation}
\end{theorem}

\bigskip

\begin{proof}
Also in this case we use the recurrence relationship (\ref{ricorrenza+}), but by considering $n = 2k+2$, $k\in \mathbb{N}$, which reads
\begin{equation}\label{ricorrenza+2}
P\lbrace \max_{0\le s\le t} \mathcal{T}(s) \in d\beta\ |\ V(0) = c, N(t) = 2k+2\rbrace  = 
\end{equation}
$$ = \int_0^{\frac{\beta}{c}} \int_{t_1}^{\frac{ct-\beta}{2c}+t_1}P\lbrace \max_{0\le s\le t-t_2} \mathcal{T}(s) \in d\beta\ |\ V(0) = c, N(t-t_2) = 2k, \mathcal{T}(0) = 2ct_1-ct_2\rbrace \cdot$$ 
$$\hspace{4cm}\cdot P\lbrace T_1 \in dt_1, T_2 \in dt_2\ | \ N(t) = 2k+2\rbrace + $$
$$ + \int_{\frac{\beta}{c}}^{\frac{ct+\beta}{2c}} P\lbrace \max_{0\le s\le t-t_2} \mathcal{T}(s) < \beta\ |\ V(0) = c, N(t-t_2) = 2k, \mathcal{T}(0) = 2\beta-ct_2\rbrace \cdot$$
$$\cdot P\lbrace T_1 \in \frac{d\beta}{c}, T_2 \in dt_2\ | \ N(t) = 2k+2\rbrace\ +\ \int_{\frac{ct+\beta}{2c}}^{t} P\lbrace T_1 \in \frac{d\beta}{c}, T_2 \in dt_2\ | \ N(t) =2k+2\rbrace $$

To prove the theorem we conjecture that
$$P\lbrace \max_{0\le s\le t} \mathcal{T}(s) \in d\beta\ |\ V(0) = c, N(t) = 2k+2\rbrace =$$
\begin{equation}
 = P\lbrace \max_{0\le s\le t} \mathcal{T}(s) \in d\beta\ |\ V(0) = c, N(t) = 2k+1\rbrace = 2\frac{(2k+1)!}{k!^2} \frac{(c^2t^2 - \beta^2)^k}{(2ct)^{2k+1}}d\beta
\end{equation}
For $k = 0,1$, that is for two and four changes of direction, this was proved by direct calculation in Orsingher (1990).
\\
Now, the first term of (\ref{ricorrenza+2}) can be written as
$$d\beta\int_0^{\frac{\beta}{c}} dt_1 \int_{t_1}^{\frac{ct-\beta}{2c}+t_1} dt_2 \frac{2(2k-1)!}{(k-1)!^2\bigl(2c(t-t_2)\bigr)^{2k-1}} \Bigl(c^2(t-t_2)^2 -(\beta-2ct_1+ct_2)^2\Bigr)^{k-1} \cdot$$
$$\hspace{6cm}\cdot \frac{(2k+2)!(t-t_2)^{2k}}{(2k)!\ t^{2k+2}} = $$ 
\begin{equation}\label{primoaddendointermedio}
= \frac{d\beta\ (2k+2)!\ c^{2k-2}}{k!(k-1)!(2ct)^{2k-1}t^3} \int_0^{\frac{\beta}{c}} dt_1 \int_{t_1}^{\frac{ct-\beta}{2c}+t_1}  \Bigl((t-t_2)^2 -(\frac{\beta}{c}-2t_1+t_2)^2\Bigr)^{k-1}(t-t_2)dt_2
\end{equation}
The inner integral in (\ref{primoaddendointermedio}) can be evaluated by using the change of variable $t_2-t_1 = w$ which leads to
$$\int_0^{\frac{ct-\beta}{2c}} \Bigl((t-t_1-w)^2 -(\frac{\beta}{c}-t_1+w)^2 \Bigr)^{k-1} (t-t_1-w) dw = $$
$$ = \int_0^{\frac{ct-\beta}{2c}} \Bigl((t-t_1)^2 -(\frac{\beta}{c}-t)^2-2w(t-2t_1+\frac{\beta}{c}) \Bigr)^{k-1} (t-t_1-w) dw = $$
$$=\bigl(t-2t_1+\frac{\beta}{c}\bigr)^{k-1}\int_0^{\frac{ct-\beta}{2c}} \bigl(t- \frac{\beta}{c}-2w  \bigr)^{k-1}  (t-t_1-w)  dw =$$
\begin{equation}\label{integraleinternop+1}
=\bigl(t-2t_1+\frac{\beta}{c}\bigr)^{k-1}  \Bigl(\frac{(t-t_1)}{2k}\bigl(t-\frac{\beta}{c}\bigr)^k - \frac{1}{2^2k(k+1)}\bigl(t-\frac{\beta}{c}\bigr)^{k+1}\Bigr)
\end{equation}
In the last step above, an integration by parts was performed.
\\By inserting result (\ref{integraleinternop+1}) into (\ref{primoaddendointermedio}) we have that the first term of (\ref{ricorrenza+2}) takes the following form
\begin{equation}\label{intermedio}
 \frac{d\beta\ (2k+2)!}{k!(k-1)!c2^{2k-1}t^{2k+2}} \int_0^\frac{\beta}{c} \bigl(t-2t_1+\frac{\beta}{c}\bigr)^{k-1}  \Bigl(\frac{(t-t_1)}{2k}\bigl(t-\frac{\beta}{c}\bigr)^k - \frac{1}{2^2k(k+1)}\bigl(t-\frac{\beta}{c}\bigr)^{k+1}\Bigr) dt_1 
\end{equation}
We now plug the following results into (\ref{intermedio})
$$\int_0^\frac{\beta}{c} \bigl(t+\frac{\beta}{c}-2t_1\bigr)^{k-1} (t-t_1) dt_1 = $$
\begin{equation}\label{aiuto+p1}
 = -\frac{1}{2k}\bigl(t-\frac{\beta}{c}\bigr)^{k+1}+\frac{t}{2k}\bigl(t+\frac{\beta}{c}\bigr)^{k}+\frac{1}{2^2k(k+1)}\bigl(t-\frac{\beta}{c}\bigr)^{k+1}-\frac{1}{2^2k(k+1)}\bigl(t+\frac{\beta}{c}\bigr)^{k+1}
\end{equation}
and
\begin{equation}\label{aiuto+p12}
\int_0^\frac{\beta}{c} \bigl(t+\frac{\beta}{c}-2t_1\bigr)^{k-1} dt_1 = \frac{1}{2k}\bigl(t+\frac{\beta}{c}\bigr)^{k} -\frac{1}{2k}\bigl(t-\frac{\beta}{c}\bigr)^{k}
\end{equation}
so that the integral (\ref{intermedio}), i.e. the first term of (\ref{ricorrenza+2}), becomes (we maintain the same order of terms in order to permit to the reader to reconstruct all details in the calculation)
$$ \frac{d\beta\ (2k+2)!}{k!(k-1)!2^{2k-1}t^{2k+2}c} \Biggl[  -\frac{1}{2^2k^2}\bigl(t-\frac{\beta}{c}\bigr)^{2k+1}+\frac{t}{2^2k^2}\bigl(t^2-\frac{\beta^2}{c^2}\bigr)^{k}+$$
$$ +\ \frac{1}{2^3k^2(k+1)}\bigl(t-\frac{\beta}{c}\bigr)^{2k+1} \ -\ \frac{1}{2^3k^2(k+1)}\bigl(t^2-\frac{\beta^2}{c^2}\bigr)^{k}\bigl(t+\frac{\beta}{c}\bigr)  +$$
$$ -\ \frac{1}{2^3k^2(k+1)}\bigl(t^2-\frac{\beta^2}{c^2}\bigr)^{k}\bigl(t-\frac{\beta}{c}\bigr)\ +\ \frac{1}{2^3k^2(k+1)}\bigl(t-\frac{\beta}{c}\bigr)^{2k+1}\Biggr] = $$
$$ =  \frac{d\beta\ (2k+2)!}{k!(k-1)!2^{2k-1}t^{2k+2}c} \Biggl[  \frac{1}{2^2k^2}\bigl(t-\frac{\beta}{c}\bigr)^{2k+1}\bigl(-1+\frac{2}{2(k+1)}\bigr) \ +  $$
$$+\ \frac{1}{2^2k^2}\bigl(t^2-\frac{\beta^2}{c^2}\bigr)^{k}\biggl(t- \frac{1}{2(k+1)} \bigl(t+\frac{\beta}{c}+t-\frac{\beta}{c}\bigr)\biggr) = $$
$$ = \frac{d\beta\ (2k+2)!}{k!(k-1)!2^{2k-1}t^{2k+2}c} \Biggl[ -\frac{1}{2^2k(k+1)}\bigl(t-\frac{\beta}{c}\bigr)^{2k+1}+   \frac{t}{2^2k(k+1)}\bigl(t^2-\frac{\beta^2}{c^2}\bigr)^{k} \Biggr]= $$
\begin{equation}\label{addendo+p1fine}
=2\frac{(2k+1)!}{k!^2} \frac{(c^2t^2 - \beta^2)^k}{(2ct)^{2k+1}}d\beta - 2^2\frac{(2k+1)!}{k!^2} \frac{(ct- \beta)^{2k+1}}{(2ct)^{2k+2}}d\beta
\end{equation}

The second integral of (\ref{ricorrenza+2}) can be developed in the following manner
$$ \int_{\frac{\beta}{c}}^{\frac{ct+\beta}{2c}} \int_0^{ct_2-\beta}  P\lbrace \max_{0\le s\le t-t_2} \mathcal{T}(s) \in dz\ |\ V(0) = c, N(t-t_2) = 2k\rbrace \cdot$$ 
$$\hspace{9cm}\cdot P\lbrace T_1 \in \frac{d\beta}{c}, T_2 \in dt_2\ | \ N(t) = 2k+2\rbrace = $$
$$ =  \int_{\frac{\beta}{c}}^{\frac{ct+\beta}{2c}} dt_2 \int_0^{ct_2-\beta} dz \bigl(c^2(t-t_2)^2- z^2\bigr)^{k-1} \frac{2(2k-1)!}{(k-1)!^2(2c)^{2k-1}(t-t_2)^{2k-1}} \cdot$$
$$\hspace{10cm}\cdot\frac{(t-t_2)^{2k}(2k+2)!}{t^{2k+2}(2k)!}  \frac{d\beta}{c} =   $$
$$ = \frac{d\beta\ (2k+2)!}{k!(k-1)!2^{2k-1}c^{2k}t^{2k+2}} \int_0^{\frac{ct-\beta}{2}} dz \int^{\frac{ct+\beta}{2c}}_{\frac{z+\beta}{c}}  \bigl(c^2(t-t_2)^2 -z^2\bigr)^{k-1}(t-t_2) dt_2= $$
$$ =  \frac{d\beta\ (2k+2)!}{k!^22^{2k}(ct)^{2k+2}} \int_0^{\frac{ct-\beta}{2}} \Biggl[\Bigl( (ct-\beta)^2 -2z(ct-\beta)\Bigr)^k- \Bigl(\bigl(\frac{ct-\beta}{2} \bigr)^2 - z^2\Bigr)^k  \Biggr] dz = $$
$$ = \frac{(2k+2)!}{k!^22^{2k}(ct)^{2k+2}}  \Biggl[\frac{(ct-\beta)^{2k+1}}{2(k+1)}-\frac{(ct-\beta)^{2k+1}}{2^{2k+2}} \int_0^1(1-w)^kw^{-\frac{1}{2}}dw \Biggr] d\beta  =$$
$$ =  \frac{(2k+2)!(ct-\beta)^{2k+1}}{k!^22^{2k}(ct)^{2k+2}} \Biggl[\frac{1}{2(k+1)} -\frac{1}{2^{2k+2}} \frac{\Gamma(k+1)\Gamma(\frac{1}{2})}{\Gamma (k+1+\frac{1}{2})} \Biggr] d\beta =$$
\begin{equation}\label{addendo+p2}
=\frac{2^2(2k+2)!(ct-\beta)^{2k+1}}{k!^2(2ct)^{2k+2}} d\beta-\frac{2^2(k+1)(ct-\beta)^{2k+1}}{(2ct)^{2k+2}}d\beta
\end{equation}
where in the last equation we applied the duplication formula of the gamma function, that is
\begin{equation}\label{duplicazionegamma}
\Gamma(2n) = \frac{\Gamma(n)\Gamma(n+\frac{1}{2}) 2^{2n-1}}{\sqrt\pi}
\end{equation}
A crucial point in the above calculation is the inversion of the order of integration which substantially simplifies the derivation of (\ref{addendo+p2}).
\\

The last term of (\ref{ricorrenza+2}) is straightforward and yields
\begin{equation}\label{addendo+p3}
\frac{d\beta}{c}  \int_{\frac{ct+\beta}{2c}}^{t} \frac{(2k+2)!(t-t_2)^{2k}}{(2k)!\ t^{2k+2}}dt_2= \frac{2^2(k+1)(ct-\beta)^{2k+1}}{(2ct)^{2k+2}}d\beta
\end{equation}
 
By summing up (\ref{addendo+p1fine}), (\ref{addendo+p2}) and (\ref{addendo+p3}) we obtain the claimed result (\ref{leggemax+p2}). 
\end{proof}

\bigskip

\begin{remark}
As a consequence of Theorem \ref{teoremamax+p} we have that, for $0<\beta <ct$
$$P \lbrace \max_{0\le s\le t} \mathcal{T}(s)  \in d \beta , N(t)\ even |\ V(0) = c\rbrace  = $$
$$ = d\beta\sum_{k=0}^\infty   2\frac{(2k+1)!}{k!^2} \frac{(c^2t^2 - \beta^2)^k}{(2ct)^{2k+1}} e^{-\lambda t}\frac{(\lambda t)^{2k+2}}{(2k+2)!} = $$
$$ = d\beta\frac{e^{-\lambda t} \lambda t}{\sqrt{c^2t^2 -\beta^2}} \sum_{k=0}^\infty \Bigl(\frac{\sqrt{c^2t^2 -\beta^2}}{2c}\Bigr)^{2k+1} \frac{1}{k!(k+1)!} = $$
\begin{equation}\label{leggecongiuntapari}
 = d\beta\frac{e^{-\lambda t} \lambda t}{\sqrt{c^2t^2 -\beta^2}} I_1\Bigl( \frac{\lambda }{c} \sqrt{c^2t^2 -\beta^2} \Bigr) 	\ =\ \frac{e^{-\lambda t}}{c}  \frac{\partial }{\partial t}  I_0\Bigl( \frac{\lambda }{c} \sqrt{c^2t^2 -\beta^2} \Bigr)d\beta
\end{equation} 
\\

From Theorem \ref{teoremamax+d} we have that, for $0<\beta <ct$
$$   P \lbrace \max_{0\le s\le t} \mathcal{T}(s)  \in d \beta, N(t)\ odd |\ V(0) = c\rbrace  =  d\beta\sum_{k=0}^\infty   2\frac{(2k+1)!}{k!^2} \frac{(c^2t^2 - \beta^2)^k}{(2ct)^{2k+1}} e^{-\lambda t}\frac{(\lambda t)^{2k+1}}{(2k+1)!} =$$
\begin{equation}\label{leggecongiuntadispari}
 = \frac{\lambda e^{-\lambda t}}{c}   I_0\Bigl( \frac{\lambda }{c} \sqrt{c^2t^2 -\beta^2} \Bigr)d\beta
\end{equation} 
\\

From (\ref{leggecongiuntapari}) and (\ref{leggecongiuntadispari}) we conclude that, for $0<\beta<ct$
$$P \lbrace \max_{0\le s\le t} \mathcal{T}(s)  \in d \beta \ |\ V(0) = c\rbrace  =  $$
\begin{equation}\label{legge+}
 = \frac{e^{-\lambda t}}{c} \Biggl[\lambda I_0\Bigl( \frac{\lambda }{c} \sqrt{c^2t^2 -\beta^2} \Bigr) + \frac{\partial }{\partial t}  I_0\Bigl( \frac{\lambda }{c} \sqrt{c^2t^2 -\beta^2} \Bigr) \Biggr]d\beta = 2P \lbrace  \mathcal{T}(t)  \in d \beta\rbrace 
\end{equation} 
\\
Thus the distribution of the maximum of the symmetric telegraph process coincides with the folded distribution of the telegraph process because a reflection principle holds in the case of the initially rightward oriented motion. Clearly, at $\beta = ct$ we have a singular component of the distribution because the absence of Poisson events brings deterministically the particle to $\beta = ct$ with probability $e^{-\lambda t}$.\hfill $\Diamond$
\end{remark}

\bigskip

\begin{remark}
It is well known that under Kac's condition, i.e. $\lambda, c\longrightarrow \infty, \ \frac{\lambda}{c^2} \longrightarrow 1$, the telegraph process $\mathcal{T}(t)$ converges to Brownian motion. Therefore, in view of (\ref{legge+}) and performing calculations similar to those displayed in Orsingher (1990), we have that the maximum of telegraph process with positive initial velocity converges to the maximum of Brownian motion.
\hfill $\Diamond$
\end{remark}

\bigskip

\section{Maximum of the initially negatively oriented telegraph process}

The most important qualitative difference between the case $V(0) = c$, treated in section 3, and $V(0) = -c$ is that the starting point can be the maximum of the process with positive probability. We must distinguish the following two cases, for $0<\beta < ct$, $k \in \mathbb{N}_0$
\begin{equation}\label{}
P\lbrace \max_{0\le s\le t} \mathcal{T}(s) \in d\beta\ |\ V(0) = -c, N(t) = 2k\rbrace 
\end{equation}
\begin{equation}\label{}
P\lbrace \max_{0\le s\le t} \mathcal{T}(s) \in d\beta\ |\ V(0) = -c, N(t) = 2k+1\rbrace 
\end{equation}

Our first result concerns the distribution conditioned on an even number of Poisson events.

\bigskip

\begin{theorem}\label{teorema-p}
Let $\lbrace \mathcal{T}(t) \rbrace_{t\ge0}$ be a symmetric telegraph process, then
\begin{equation}\label{legge-p}
P\lbrace \max_{0\le s\le t} \mathcal{T}(s) \in d\beta\ |\ V(0) = -c, N(t) = 2k\rbrace \ =\ \frac{2(2k)!}{k!(k-1)!} \frac{(c^2t^2 - \beta^2)^{k-1}(ct-\beta)}{(2ct)^{2k}} d\beta
\end{equation}
for $0<\beta<ct$, $k\ge 1$ and
\begin{equation}\label{singolarita-p}
  P\lbrace \max_{0\le s\le t} \mathcal{T}(s) = 0\ |\ V(0) = -c, N(t) = 2k\rbrace = \binom{2k}{k}\frac{1}{2^{2k}}
\end{equation}
for $k\ge 0$.
\end{theorem}

\bigskip

\begin{proof}
We must consider that at time $T_1$ the telegraph particle starts moving rightward and thus in the residual time lapse $(T_1, t)$ behaves as in section 3. We must distinguish two cases:

- Case 1. where 
\begin{equation}\label{caso1}
 \beta -(-cT_1) \le c(T_1 - t)
\end{equation}
so the moving particle has enough time to cross the starting point during $(T_1, t)$;

- Case 2. where 
\begin{equation}\label{caso2}
 \beta -(-cT_1) > c(T_1 - t)
\end{equation}
so that the moving particle gets so far on the negative half-line with the first displacement that will never be able to reach level $\beta$.
\\
Thus, if case (\ref{caso1}). occurs we have that
\begin{equation}\label{passo1-}
 P\lbrace \max_{0\le s\le t} \mathcal{T}(s) \in d\beta\ |\ V(0) = -c, N(t) = 2k\rbrace =
\end{equation}
$$ = \int_0^{\frac{ct-\beta}{2c}}  P\lbrace \max_{0\le s\le t-t_1} \mathcal{T}(s) \in d\beta\ |\ V(0) = c, N(t-t_1) = 2k-1, \mathcal{T}(0) = -ct_1\rbrace \cdot$$
$$ \hspace{6cm} \cdot P\lbrace T_1 \in dt_1 \ | \ N(t) = 2k\rbrace = $$
$$ = \int_0^{\frac{ct-\beta}{2c}}  \frac{2(2k-1)!}{(k-1)!^2\bigl(2c(t-t_1)\bigr)^{2k-1}} \Bigl(c^2(t-t_1)^2 -(\beta +ct_1)^2\Bigr)^{k-1}  \frac{2k}{t^{2k}}(t-t_1)^{2k-1} dt_1$$
In the last step we applied result (\ref{leggemax+d2}) of Theorem \ref{teoremamax+d}, suitably adapted to the framework (\ref{passo1-}), i.e. we replace $2k+1$ by $2k-1$ and $t$ by $t-t_1$. In conclusion, we have that
$$P\lbrace \max_{0\le s\le t} \mathcal{T}(s) \in d\beta\ |\ V(0) = -c, N(t) = 2k\rbrace =$$
$$ =  \frac{d\beta \ 2(2k)!}{(k-1)!^2(2c)^{2k-1}t^{2k}} \int_0^{\frac{ct-\beta}{2c}} \Bigl(c^2t^2-\beta^2-2ct_1(ct+\beta)\Bigr)^{k-1} dt_1= $$
$$   = 2\frac{(2k)!}{k!(k-1)!} \frac{(c^2t^2 - \beta^2)^{k-1}(ct-\beta)}{(2ct)^{2k}} d\beta $$
and this confirms formula (\ref{legge-p}).
\\In order to obtain (\ref{singolarita-p}) we observe that
\begin{equation}\label{passo2-}
\int_0^{ct} (c^2t^2 - \beta^2)^{k-1}(ct-\beta) d\beta = \frac{(ct)^{2k}}{2} \int_0^1 (1-\sqrt w)(1-w)^{k-1} \frac{dw}{\sqrt w} =\frac{(ct)^{2k}}{2}  \Biggl( \frac{\Gamma(k)\Gamma(\frac{1}{2})}{\Gamma (k + \frac{1}{2})} -\frac{1}{k} \Biggr)
\end{equation}
\\
In view of (\ref{passo2-}) we have therefore that
$$ \int_0^{ct} P\lbrace \max_{0\le s\le t} \mathcal{T}(s) \in d\beta\ |\ V(0) = -c, N(t) = 2k\rbrace = $$
$$ =  2\frac{(2k)!}{k!(k-1)!} \frac{(ct)^{2k}}{2 (2ct)^{2k}}  \Biggl( \frac{(k-1)!\sqrt \pi (k-1)!}{\sqrt \pi 2^{1-2k} (2k-1)!} -\frac{1}{k} \Biggr) = 1-\frac{(2k)!}{k!^2}\frac{1}{2^{2k}} =$$ 
$$ =  1- P\lbrace \max_{0\le s\le t} \mathcal{T}(s) =0\ |\ V(0) = -c, N(t) = 2k\rbrace$$
as claimed in (\ref{singolarita-p}). 
\end{proof}

\bigskip

\begin{remark}\label{sing}
The singularity
$$  P\lbrace \max_{0\le s\le t} \mathcal{T}(s) =0\ |\ V(0) = -c, N(t) = 2k\rbrace = \frac{(2k)!}{k!^2}\frac{1}{2^{2k}} $$
for large values of $k$ decreases as $\frac{1}{\sqrt{ \pi k}}$ as the application of Stirling's formula shows.
\hfill $\Diamond$
\end{remark}

\bigskip

\begin{theorem}\label{teorema-d}
Let $\lbrace \mathcal{T}(t) \rbrace_{t\ge0}$ be a symmetric telegraph process, then
$$P\lbrace \max_{0\le s\le t} \mathcal{T}(s) \in d\beta\ |\ V(0) = -c, N(t) = 2k+1\rbrace =$$
\begin{equation}\label{legge-d}
  =\Biggl( \frac{(2k+1)!}{(k-1)!(k+1)!} \frac{(c^2t^2 - \beta^2)^{k-1}(ct-\beta)}{(2ct)^{2k}} +\frac{(2k+1)!}{k!(k+1)!} \frac{(c^2t^2 - \beta^2)^{k}}{(2ct)^{2k+1}}  \Biggr)d\beta 
\end{equation}
for $0<\beta<ct$, $k \in \mathbb{N}_0$ and
\begin{equation}\label{singolarita-d}
  P\lbrace \max_{0\le s\le t} \mathcal{T}(s) = 0\ |\ V(0) = -c, N(t) = 2k+1\rbrace = \binom{2k+1}{k}\frac{1}{2^{2k+1}}
\end{equation}
\end{theorem}

\bigskip

\begin{proof}
Also here we must distinguish the two cases (\ref{caso1}) and (\ref{caso2}). In the first one we have 
$$P\lbrace \max_{0\le s\le t} \mathcal{T}(s) \in d\beta\ |\ V(0) = -c, N(t) = 2k+1\rbrace =$$
$$ = \int_0^{\frac{ct-\beta}{2c}}  P\lbrace \max_{0\le s\le t-t_1} \mathcal{T}(s) \in d\beta\ |\ V(0) = c, N(t-t_1) = 2k, \mathcal{T}(0) = -ct_1\rbrace \cdot$$
$$ \hspace{6cm} \cdot P\lbrace T_1 \in dt_1 \ | \ N(t) = 2k+1\rbrace = $$
(by applying (\ref{leggemax+p2}) )
$$ = d\beta\int_0^{\frac{ct-\beta}{2c}}  \frac{2(2k-1)!}{(k-1)!^2\bigl(2c(t-t_1)\bigr)^{2k-1}} \Bigl(c^2(t-t_1)^2 -(\beta +ct_1)^2\Bigr)^{k-1}  \frac{2k+1}{t^{2k+1}}(t-t_1)^{2k} dt_1 = $$
$$ =   \frac{d\beta\ (2k+1)! (ct+\beta)^{k-1}}{k!(k-1)!(2c)^{2k-1}t^{2k+1}}  \int_0^{\frac{ct-\beta}{2c}}  (ct-\beta -2ct_1)^{k-1} (t-t_1) dt_1 =$$
$$ =    \frac{(2k+1)! (ct+\beta)^{k-1}}{k!(k-1)!(2c)^{2k-1}t^{2k+1}} \Biggl( \frac{t(ct-\beta)^k}{2ck} - \frac{(ct-\beta)^{k+1}}{(2c)^2k(k+1)}\Biggr)d\beta =  $$
$$ =  \frac{ (2k+1)!(c^2t^2-\beta^2)^{k-1}(ct-\beta)}{k!^2(2ct)^{2k}}  \Biggl( 1 - \frac{ct-\beta}{2ct(k+1)}\Biggr)  d\beta =$$
$$ =  \frac{ (2k+1)!(c^2t^2-\beta^2)^{k-1}(ct-\beta)}{k!^2(2ct)^{2k}}  \Biggl( \frac{k}{k+1} + \frac{ct+\beta}{2ct(k+1)}\Biggr)  d\beta =$$
$$  =\Biggl( \frac{(2k+1)!}{(k-1)!(k+1)!} \frac{(c^2t^2 - \beta^2)^{k-1}(ct-\beta)}{(2ct)^{2k}} +\frac{(2k+1)!}{k!(k+1)!} \frac{(c^2t^2 - \beta^2)^{k}}{(2ct)^{2k+1}}  \Biggr)d\beta $$
\\

For the evaluation of the singular component of the distribution we need (\ref{passo2-}) and
$$\int_0^{ct} (c^2t^2 -\beta^2)^{k} d\beta = \frac{(ct)^{2k+1}}{2} \frac{\Gamma(k+1)\Gamma(\frac{1}{2})}{\Gamma (k +1+ \frac{1}{2})}$$
Thus
$$ P\lbrace \max_{0\le s\le t} \mathcal{T}(s) = 0\ |\ V(0) = -c, N(t) = 2k+1\rbrace =$$
$$ = 1-\Biggl[\frac{(2k+1)!}{(k-1)!(k+1)!} \frac{(ct)^{2k}}{2(2ct)^{2k}}  \Biggl(\frac{(k-1)!\sqrt \pi (k-1)!}{\sqrt \pi 2^{1-2k} (2k-1)!} -\frac{1}{k}\Biggr) \  +$$
$$+\ \frac{(2k+1)!}{k!(k+1)!}   \frac{(ct)^{2k+1}}{2(2ct)^{2k+1}}\frac{k!\sqrt \pi k!}{\sqrt \pi 2^{1-2k-2} (2k+1)!}  \Biggr] = $$
$$ =1-\Biggl[ \frac{2k+1}{2k+2} - \frac{(2k+1)!}{k!(k+1)!}\frac{1}{2^{2k+1}} + \frac{1}{2k+2}\Biggr] = \binom{2k+1}{k}\frac{1}{2^{2k+1}} $$
\end{proof}

\bigskip

\begin{remark}
We notice that the probability mass of the singularity of the conditioned telegraph process starting with a negative direction depends on the number of changes of direction only, so it is independent of both time $t$ and speed $c$. 
\\Furthermore, we have that
$$P\lbrace \max_{0\le s\le t} \mathcal{T}(s) = 0\ |\ V(0) = -c, N(t) = 2k+1\rbrace = \frac{(2k+1)!}{k!(k+1)!}\frac{1}{2^{2k+1}} \cdot \frac{2k+2}{2(k+1)} = $$
$$ = P\lbrace \max_{0\le s\le t} \mathcal{T}(s) = 0\ |\ V(0) = -c, N(t) = 2k+2\rbrace$$
and
$$ P\lbrace \max_{0\le s\le t} \mathcal{T}(s) = 0\ |\ V(0) = -c, N(t) = 2k\rbrace = \frac{(2k)!}{k!^2}\frac{1}{2^{2k}}\ >\ \frac{(2k)!}{k!^2}\frac{1}{2^{2k}} \cdot \frac{2k+1}{2k+2}\ =$$
\begin{equation}\label{sing2}
=\ P\lbrace \max_{0\le s\le t} \mathcal{T}(s) = 0\ |\ V(0) = -c, N(t) = 2k+1\rbrace
\end{equation}
Therefore, we can extend the statement of Remark \ref{sing}.\hfill $\Diamond$
\end{remark}

\bigskip

\begin{remark}
In view of (\ref{leggemax+p2}) and (\ref{legge-p}) we can also write, for $0 < \beta < ct$
$$P\lbrace \max_{0\le s\le t} \mathcal{T}(s) \in d\beta\ |\ V(0) = -c, N(t) = 2k+1\rbrace =$$
$$  =\Biggl( \frac{(2k+1)!}{(k-1)!(k+1)!} \frac{(c^2t^2 - \beta^2)^{k-1}(ct-\beta)}{(2ct)^{2k}} +\frac{(2k+1)!}{k!(k+1)!} \frac{(c^2t^2 - \beta^2)^{k}}{(2ct)^{2k+1}}  \Biggr)d\beta =  $$
$$=\frac{2k+1}{2k+2} P\lbrace \max_{0\le s\le t} \mathcal{T}(s) \in d\beta\ |\ V(0) = -c, N(t) = 2k\rbrace\ +$$
\begin{equation}\label{sommapesata}
+\ \frac{1}{2k+2}P\lbrace \max_{0\le s\le t} \mathcal{T}(s) \in d\beta\ |\ V(0) = c, N(t) = 2k+1\rbrace
\end{equation}
\\
The last relationship shows that the maximum of $\mathcal{T}(t)$ under the condition $V(0) = -c, N(t) = 2k+1$ is a weighted sum of two previously obtained distributions, with a prevailing weight for the case of the leftward starting motion. 
\hfill $\Diamond$
\end{remark}

\bigskip

\begin{remark}
For the unconditional distributions we have that, for $0 < \beta < ct$
$$P \lbrace \max_{0\le s\le t} \mathcal{T}(s)  \in d \beta, N(t)\ even |\ V(0) = -c\rbrace  = $$
$$=  d\beta\sum_{k=1}^\infty \frac{2(2k)!}{k!(k-1)!} \frac{(c^2t^2 - \beta^2)^{k-1}(ct-\beta)}{(2ct)^{2k}} e^{-\lambda t}\frac{(\lambda t)^{2k}}{(2k)!} = $$
$$ = d\beta\frac{(ct-\beta)e^{-\lambda t} \lambda}{c\sqrt{c^2t^2 -\beta^2}} \sum_{k=0}^\infty \Bigl(\frac{\sqrt{c^2t^2 -\beta^2}}{2c}\Bigr)^{2k+1} \frac{1}{k!(k+1)!} = $$
\begin{equation}\label{leggecongiuntapari-}
 =e^{-\lambda t}\frac{ \lambda(ct-\beta)}{c\sqrt{c^2t^2 -\beta^2}} I_1\Bigl( \frac{\lambda }{c} \sqrt{c^2t^2 -\beta^2} \Bigr) d\beta \ =\ \frac{ e^{-\lambda t}}{c} \Bigl( 1-\frac{\beta}{ct}\Bigr)\frac{\partial }{\partial t}  I_0\Bigl( \frac{\lambda }{c} \sqrt{c^2t^2 -\beta^2} \Bigr)d\beta
\end{equation} 
\\
For the other initially left-oriented motion we have
$$P \lbrace \max_{0\le s\le t} \mathcal{T}(s)  \in d \beta , N(t)\ odd |\ V(0) = -c\rbrace  = $$
$$=  d\beta\sum_{k=0}^\infty  \Biggl( \frac{(2k+1)!}{(k-1)!(k+1)!} \frac{(c^2t^2 - \beta^2)^{k-1}(ct-\beta)}{(2ct)^{2k}} +\frac{(2k+1)!}{k!(k+1)!} \frac{(c^2t^2 - \beta^2)^{k}}{(2ct)^{2k+1}}  \Biggr) e^{-\lambda t}\frac{(\lambda t)^{2k+1}}{(2k+1)!} =$$
$$ =  e^{-\lambda t} \Biggl[ \sum_{k=0}^\infty  \frac{(c^2t^2 - \beta^2)^{k}}{k!(k+1)!} \Bigl( \frac{\lambda}{2c}\Bigr)^{2k+1} +(ct-\beta) \lambda t\sum_{k=1}^\infty\frac{(c^2t^2 - \beta^2)^{k-1}}{(k-1)!(k+1)!} \Bigl( \frac{\lambda}{2c}\Bigr)^{2k}  \Biggr] d\beta=  $$
$$  =  e^{-\lambda t} \Biggl[ \frac{1}{\sqrt{c^2t^2-\beta^2}}I_1\Bigl( \frac{\lambda }{c} \sqrt{c^2t^2 -\beta^2} \Bigr) + \frac{\lambda t(ct-\beta)}{c^2t^2-\beta^2}\sum_{k=0}^\infty \Bigl(  \frac{\lambda }{2c} \sqrt{c^2t^2 -\beta^2}\Bigr)^{2k+2}\frac{1}{k!\Gamma(k+1+2)} \Biggr] d\beta = $$
\begin{equation}\label{congiuntadisp-}
 = e^{-\lambda t} \Biggl[ \frac{1}{\sqrt{c^2t^2-\beta^2}}I_1\Bigl( \frac{\lambda }{c} \sqrt{c^2t^2 -\beta^2} \Bigr) + \frac{\lambda t(ct-\beta)}{c^2t^2-\beta^2} I_2\Bigl( \frac{\lambda }{c} \sqrt{c^2t^2 -\beta^2} \Bigr) \Biggr] d\beta 
\end{equation}
\\
The well-known relationship, with $r$ integer
\begin{equation}\label{bessel}
I_{r+1}(x) = I_{r-1}(x) - \frac{2r}{x}I_r(x)
\end{equation}
for $r=1$ permits us to simplify (\ref{congiuntadisp-}) as
$$P \lbrace \max_{0\le s\le t} \mathcal{T}(s)  \in d \beta, N(t)\ odd |\ V(0) = -c\rbrace  =$$
$$ = \frac{e^{-\lambda t}}{\sqrt{c^2t^2-\beta^2}} I_1\Bigl( \frac{\lambda }{c} \sqrt{c^2t^2 -\beta^2} \Bigr)d\beta \  +e^{-\lambda t}\frac{\lambda t(ct-\beta)}{c^2t^2-\beta^2} \Biggl[ I_0\Bigl( \frac{\lambda }{c} \sqrt{c^2t^2 -\beta^2} \Bigr) -2 \frac{I_1\Bigl( \frac{\lambda }{c} \sqrt{c^2t^2 -\beta^2} \Bigr)}{\frac{\lambda}{c}  \sqrt{c^2t^2 -\beta^2}} \Biggr] d\beta = $$ 
$$ =  \frac{e^{-\lambda t}}{\sqrt{c^2t^2-\beta^2}} \Bigl(1- \frac{2ct(ct-\beta)}{c^2t^2-\beta^2}\Bigr)I_1\Bigl( \frac{\lambda }{c} \sqrt{c^2t^2 -\beta^2} \Bigr)d\beta \ +e^{-\lambda t}\frac{\lambda t(ct-\beta)}{c^2t^2-\beta^2}I_0\Bigl( \frac{\lambda }{c} \sqrt{c^2t^2 -\beta^2} \Bigr)d\beta =$$
$$ = \frac{e^{-\lambda t}}{\sqrt{c^2t^2-\beta^2}}\Bigl(- \frac{(ct-\beta)^2}{c^2t^2-\beta^2}\Bigr)I_1\Bigl( \frac{\lambda }{c} \sqrt{c^2t^2 -\beta^2} \Bigr)d\beta +\frac{e^{-\lambda t}\lambda t}{ct+\beta}I_0\Bigl( \frac{\lambda }{c} \sqrt{c^2t^2 -\beta^2} \Bigr)d\beta = $$ 
\begin{equation}\label{congdisp}
 = \frac{e^{-\lambda t}}{ct+\beta} \Biggl[\lambda t I_0\Bigl( \frac{\lambda }{c} \sqrt{c^2t^2 -\beta^2} \Bigr) - \sqrt{\frac{ct-\beta}{ct+\beta}}I_1\Bigl( \frac{\lambda }{c} \sqrt{c^2t^2 -\beta^2} \Bigr)\Biggr] d\beta =
\end{equation}
$$ = \frac{ct}{ct+\beta} P \lbrace \max_{0\le s\le t} \mathcal{T}(s)  \in d \beta , N(t)\ odd |\ V(0) = c\rbrace \ +$$
$$-\ \frac{c}{\lambda(ct+\beta)} P \lbrace \max_{0\le s\le t} \mathcal{T}(s)  \in d \beta , N(t)\ even |\ V(0) = -c\rbrace $$
where in the last step we applied formulas (\ref{leggecongiuntadispari}) and (\ref{leggecongiuntapari-}).
\hfill $\Diamond$
\end{remark}

\bigskip

\begin{remark}
At last, we also have that, for $0 < \beta < ct$
$$P \lbrace \max_{0\le s\le t} \mathcal{T}(s)  \in d \beta\ |\ V(0) = -c\rbrace  =$$ 
\begin{equation}\label{max-}
= e^{-\lambda t}\Biggl[  \frac{\lambda t
}{ct+\beta}   I_0\Bigl( \frac{\lambda }{c} \sqrt{c^2t^2 -\beta^2} \Bigr) +\sqrt{\frac{ct-\beta}{ct+\beta}}\Bigl(\frac{\lambda}{c} -\frac{1}{ct+\beta} \Bigr)I_1\Bigl( \frac{\lambda }{c} \sqrt{c^2t^2 -\beta^2} \Bigr)\Biggr ]d\beta
\end{equation}
and
\begin{equation}\label{maxsing-}
P \lbrace \max_{0\le s\le t} \mathcal{T}(s) =0\ |\ V(0) = -c\rbrace  \ =\ e^{-\lambda t}\Biggl[ I_0\Bigl( \lambda t \Bigr) +I_1\Bigl( \lambda t \Bigr)\Biggr ]
\end{equation}
\\
By considering the aymptotic estimate of the Bessel functions, $x \longrightarrow \infty$
$$ I_{\nu} (x) \sim \frac{e^{x}}{\sqrt{2\pi x}}$$
for $\nu  = 0,1,2 $; it is possible to check that under Kac's condition the maximum of  the telegraph process with $V(0) = -c$ converges to the maximum of Brownian motion.
Hence, if $\lbrace B(t) \rbrace_{t\ge 0}$ is a standard Brownian motion, we claim that, under Kac's condition 
$$\max_{0\le s\le t} \mathcal{T}(s)  \stackrel {d}{\longrightarrow } \max_{0\le s\le t}  B(s) $$
\hfill $\Diamond$
\end{remark}

\bigskip

\section{Some special cases}

From (\ref{leggemax+p2}) we can derive the distribution of the maximum of the initially right-oriented telegraph process. We state this result in the next theorem.

\begin{corollary}
Let $\lbrace \mathcal{T}(t) \rbrace_{t\ge0}$ be a symmetric telegraph process, then
$$P\lbrace \max_{0\le s\le t} \mathcal{T}(s) < \beta\ |\ V(0) = c, N(t) = 2k+2\rbrace = $$
\begin{equation}\label{rip+d}
=P\lbrace \max_{0\le s\le t} \mathcal{T}(s) < \beta\ |\ V(0) = c, N(t) = 2k+1\rbrace =\frac{\beta}{ct} \sum_{j=0}^k \binom{2j}{j}\Bigl( \frac{\sqrt{c^2t^2-\beta^2}}{2ct}\Bigr)^{2j} 
\end{equation}
for $0<\beta < ct,\ k \in \mathbb{N}_0$.
\end{corollary}

\begin{proof}
We begin by evaluating the following integral:
$$ \int_0^\beta (c^2t^2 - x^2)^k dx = \ \ \ \ \ \Bigl(x = ct\sin \phi \Bigr)$$
$$ = (ct)^{2k+1}  \int_0^{\arcsin\frac{\beta}{ct} } (\cos\phi)^{2k+1} d\phi $$
By integrating by parts we obtain the recurrence formula
\begin{equation}\label{intcoseno}
I_{2k+1} = \int_0^{\arcsin\frac{\beta}{ct} } (\cos\phi)^{2k+1} d\phi = \Bigl( 1-\frac{\beta^2}{c^2t^2}\Bigr)^k \frac{\beta}{ct(2k+1)} + \frac{2k}{2k+1} I_{2k-1}
\end{equation}
\\
Thus
$$ \int_0^\beta (c^2t^2 - x^2)^k dx =  (ct)^{2k+1}  \frac{\beta}{ct} \Biggl( \Bigl( 1-\frac{\beta^2}{c^2t^2}\Bigr)^k \frac{1}{(2k+1)} + \Bigl( 1-\frac{\beta^2}{c^2t^2}\Bigr)^{k-1 }\frac{2k}{(2k+1)(2k-1)} +$$
$$ +\  \Bigl( 1-\frac{\beta^2}{c^2t^2}\Bigr)^{k-2 }\frac{(2k)(2k-2)}{(2k+1)(2k-1)(2k-3) } +\ .\ .\ .\  + $$
$$ + \Bigl( 1-\frac{\beta^2}{c^2t^2}\Bigr)^{k-j }\frac{(2k)(2k-2)\cdot \cdot \cdot (2k-2j+2)}{(2k+1)(2k-1)\cdot \cdot \cdot (2k-2j+1) } + \ .\ .\ .\ + \frac{(2k)(2k-2)\cdot \cdot \cdot 2}{(2k+1)(2k-1)\cdot \cdot \cdot 1 }  \Biggr) = $$
$$ = (ct)^{2k}\beta \sum_{j=0}^k \Bigl(  1-\frac{\beta^2}{c^2t^2}\Bigr)^{k-j} \Bigl(\frac{2^j k!}{(k-j)!} \Bigr)^2 \frac{(2k-2j)!}{(2k+1)!} = $$
$$ = \beta\frac{(ct)^{2k}k!^2}{(2k+1)!} \sum_{j=0}^k \Bigl(  \frac{c^2t^2-\beta^2}{c^2t^2}\Bigr)^{k-j} 2^{2j} \frac{(2k-2j)!}{(k-j)!^2} = $$
\begin{equation}\label{intperrip}
 =  \beta\frac{(2ct)^{2k}k!^2}{(2k+1)!} \sum_{j=0}^k \Bigl( \frac{c^2t^2-\beta^2}{c^2t^2}\Bigr)^j\frac{1}{2^{2j}}\frac{(2j)!}{j!^2} 
\end{equation}
\\
In view of (\ref{intperrip}) we have
$$P\lbrace \max_{0\le s\le t} \mathcal{T}(s) < \beta\ |\ V(0) = c, N(t) = 2k+2\rbrace  = P\lbrace \max_{0\le s\le t} \mathcal{T}(s) < \beta\ |\ V(0) = c, N(t) = 2k+1\rbrace = $$
$$ = \int_0^\beta \frac{2(2k+1)!}{k!^2} \frac{(c^2t^2 - x^2)^k}{(2ct)^{2k+1}}dx =$$
$$ =  \frac{2(2k+1)!}{k!^2(2ct)^{2k+1}}  \beta\frac{(2ct)^{2k}k!^2}{(2k+1)!} \sum_{j=0}^k \Bigl( \frac{c^2t^2-\beta^2}{c^2t^2}\Bigr)^j\frac{1}{2^{2j}}\frac{(2j)!}{j!^2} =  $$
$$ = \frac{\beta}{ct} \sum_{j=0}^k \Bigl( \frac{c^2t^2-\beta^2}{c^2t^2}\Bigr)^j \frac{1}{2^{2j}}\binom{2j}{j} $$
\end{proof}

\bigskip

From  formula (\ref{rip+d}) we obtain that, for $0<\beta<ct$
$$P\lbrace \max_{0\le s\le t} \mathcal{T}(s) < \beta\ |\ V(0) = c, N(t) =1\rbrace  = \frac{\beta}{ct}$$
$$P\lbrace \max_{0\le s\le t} \mathcal{T}(s) < \beta\ |\ V(0) = c, N(t) =3\rbrace  = \frac{\beta}{2c^3t^3}(3c^2t^2-\beta^2) $$
$$P\lbrace \max_{0\le s\le t} \mathcal{T}(s) < \beta\ |\ V(0) = c, N(t) =5\rbrace  = \frac{\beta}{(ct)^5}\Bigl[ 5(c^2t^2-\beta^2)(3c^2t^2+\beta^2)+ \beta^4\Bigr]$$
\\
These formulas coincide with those reported in table 1 of Orsingher (1990).

\bigskip

Analogously we can find the distribution function related to (\ref{legge-p}).

\begin{corollary}
Let $\lbrace \mathcal{T}(t) \rbrace_{t\ge0}$ be a symmetric telegraph process, then
$$P\lbrace \max_{0\le s\le t} \mathcal{T}(s) < \beta\ |\ V(0) = -c, N(t) = 2k\rbrace = $$
\begin{equation}\label{rip-p}
 =\frac{\beta}{ct} \sum_{j=0}^{k-1} \Bigl( \frac{c^2t^2-\beta^2}{c^2t^2}\Bigr)^j \frac{1}{2^{2j}}\binom{2j}{j} +  \frac{(c^2t^2-\beta^2)^k }{(2ct)^{2k}}\binom{2k}{k}
\end{equation}
for $0<\beta < ct,\ k \in \mathbb{N}$.
\end{corollary}
We note that the first term of (\ref{rip-d}) coincides with the cumulative distribution function 
\\$P\lbrace \max_{0\le s\le t} \mathcal{T}(s) < \beta\ |\ V(0) = c, N(t) = 2k-1\rbrace $.

\bigskip

\begin{proof}
$$P\lbrace \max_{0\le s\le t} \mathcal{T}(s) < \beta\ |\ V(0) =- c, N(t) = 2k\rbrace  = $$
$$=P\lbrace \max_{0\le s\le t} \mathcal{T}(s) =0\ |\ V(0) =- c, N(t) = 2k\rbrace  + \frac{2(2k)!}{k!(k-1)!(2ct)^{2k}} \int_0^\beta (c^2t^2 - x^2)^{k-1}(ct-x) dx =$$
$$ = \binom{2k}{k}\frac{1}{2^{2k}} +  \frac{2(2k)!}{k!(k-1)!(2ct)^{2k}} \Biggl( ct \int_0^\beta (c^2t^2 - x^2)^{k-1} dx -\int_0^\beta x(c^2t^2 - x^2)^{k-1} dx \Biggr)  = $$
(by applying formula (\ref{intperrip}) suitably adapted)
$$  =  \binom{2k}{k}\frac{1}{2^{2k}} +  \frac{(2k)!}{k!(k-1)!(2ct)^{2k-1}}  \beta\frac{(2ct)^{2k-2}(k-1)!^2}{(2k-1)!} \sum_{j=0}^{k-1} \Bigl( \frac{c^2t^2-\beta^2}{c^2t^2}\Bigr)^j \frac{1}{2^{2j}}\binom{2j}{j}  + $$
$$   +  \frac{2(2k)!}{k!(k-1)!(2ct)^{2k}} \Bigl( \frac{(c^2t^2-\beta^2)^k - (ct)^{2k}}{2k}\Bigr) =$$
$$ = \binom{2k}{k}\frac{1}{2^{2k}} + \frac{\beta}{ct} \sum_{j=0}^{k-1} \Bigl( \frac{c^2t^2-\beta^2}{c^2t^2}\Bigr)^j \frac{1}{2^{2j}}\binom{2j}{j} +  \frac{(2k)!(c^2t^2-\beta^2)^k }{k!^2 (2ct)^{2k}} -  \frac{(2k)!}{k!k!2^{2k}} = $$ 
$$ = \frac{\beta}{ct} \sum_{j=0}^{k-1} \Bigl( \frac{c^2t^2-\beta^2}{c^2t^2}\Bigr)^j \frac{1}{2^{2j}}\binom{2j}{j} +  \frac{(c^2t^2-\beta^2)^k }{(2ct)^{2k}}\binom{2k}{k}$$
\end{proof}

For $k = 1,2$ we have that
$$P\lbrace \max_{0\le s\le t} \mathcal{T}(s) < \beta\ |\ V(0) = -c, N(t) = 2\rbrace = \frac{\beta}{ct} + \frac{(c^2t^2-\beta^2)}{2c^2t^2}$$
$$P\lbrace \max_{0\le s\le t} \mathcal{T}(s) < \beta\ |\ V(0) = -c, N(t) = 4\rbrace = $$
$$ = \frac{\beta}{ct}\Bigl(1+ \frac{c^2t^2- \beta^2}{2c^2t^2}\Bigr) + \frac{3}{2^3}\Bigl( \frac{c^2t^2- \beta^2}{c^2t^2}\Bigr)^2 = \frac{1}{2^3c^3t^3}\Bigl[ 3(c^2t^2-\beta^2)^2 +4ct\beta(3c^2t^2 -\beta^2) \Bigr]$$
which coincide with the corresponding formulas of table 3 of Orsingher (1990), for $c = 1$.

\bigskip

\begin{corollary}
Let $\lbrace \mathcal{T}(t) \rbrace_{t\ge0}$ be a symmetric telegraph process, then
$$P\lbrace \max_{0\le s\le t} \mathcal{T}(s) < \beta\ |\ V(0) = -c, N(t) = 2k+1\rbrace = $$
\begin{equation}\label{rip-d}
=\frac{\beta}{ct} \sum_{j=0}^{k-1} \Bigl( \frac{c^2t^2-\beta^2}{c^2t^2}\Bigr)^j \frac{1}{2^{2j}}\binom{2j}{j} +  \frac{(2k)!}{k!(k+1)!}\frac{(c^2t^2-\beta^2)^k }{2(2ct)^{2k}}\Bigl(\frac{\beta}{ct} +(2k+1) \Bigr)
\end{equation}
	for $0<\beta < ct,\ k \in \mathbb{N}_0$.
\end{corollary}
We note that the first term of (\ref{rip-d}) coincides with the cumulative distribution function 
\\$P\lbrace \max_{0\le s\le t} \mathcal{T}(s) < \beta\ |\ V(0) = c, N(t) = 2k-1\rbrace $.

\bigskip

\begin{proof}
By means of the relationships (\ref{sing2}) and (\ref{sommapesata}) we immediatly obtain that
$$ P\lbrace \max_{0\le s\le t} \mathcal{T}(s) < \beta\ |\ V(0) = -c, N(t) = 2k+1\rbrace =  $$
$$=\frac{2k+1}{2k+2} P\lbrace \max_{0\le s\le t} \mathcal{T}(s) < \beta\ |\ V(0) = -c, N(t) = 2k\rbrace\ +$$
$$+\ \frac{1}{2k+2}P\lbrace \max_{0\le s\le t} \mathcal{T}(s) < \beta\ |\ V(0) = c, N(t) = 2k+1\rbrace$$
and the thesis follows by considering formulas (\ref{rip+d}) and (\ref{rip-p}).
\end{proof}

\bigskip

\begin{remark}
We have the following $m$-th order moments of the conditional distributions
$$\mathbb{E} \Biggl[ \Bigl(\max_{0\le s\le t} \mathcal{T}(s)\Bigr)^m \ |\ V(0) = c, N(t) = 2k+1\Biggr] = \mathbb{E} \Biggl[ \Bigl(\max_{0\le s\le t} \mathcal{T}(s)\Bigr)^m \ |\ V(0) = c, N(t) = 2k+2\Biggr] =$$
\begin{equation}\label{momento+d}
= \frac{(2k+1)!(ct)^m}{2^{2k+1}k!} \frac{\Gamma\Bigl(\frac{m+1}{2}\Bigr)}{\Gamma\Bigl(k+1+\frac{m+1}{2}\Bigr)}  
\end{equation}

In particular, from (\ref{momento+d}) we have that
\begin{equation}\label{mom1+d}
\mathbb{E} \Biggl[ \max_{0\le s\le t} \mathcal{T}(s) \ |\ V(0) = c, N(t) = 2k+1\Biggr] = \binom{2k+1}{k}\frac{ct}{2^{2k+1}}
\end{equation}
\begin{equation}\label{momsec+d}
\mathbb{E} \Biggl[ \Bigl(\max_{0\le s\le t} \mathcal{T}(s)\Bigr)^2 \ |\ V(0) = c, N(t) = 2k+1\Biggr] = \frac{c^2t^2}{2k+3}
\end{equation}
Therefore, the maximum of the conditioned telegraph process, starting with a positive direction, converges in mean square to $0 \ a.s.$ as $k \longrightarrow \infty$.

For $k=0,1,2$ the results of formula (\ref{mom1+d}) coincide with the mean-values presented in table 2 of Orsingher (1990), thus
$$ \mathbb{E} \Biggl[ \max_{0\le s\le t} \mathcal{T}(s) \ |\ V(0) = c, N(t) = 1\Biggr] = \frac{2ct}{2^2} \ \ \  ;\ \ \  \mathbb{E} \Biggl[ \max_{0\le s\le t} \mathcal{T}(s) \ |\ V(0) = c, N(t) = 3\Biggr] = \frac{3ct}{2^3}$$
$$ \mathbb{E} \Biggl[ \max_{0\le s\le t} \mathcal{T}(s) \ |\ V(0) = c, N(t) = 5\Biggr] = \frac{5ct}{2^4} $$
\\

For the unconditional initially positively oriented telegraph process we have, $m\ge 1$
\begin{equation}\label{momento+}
\mathbb{E} \Biggl[ \Bigl(\max_{0\le s\le t} \mathcal{T}(s)\Bigr)^m \ |\ V(0) = c\Biggr] =  e^{-\lambda t}(ct)^m\Bigl(\frac{2}{\lambda t}\Bigr)^{\frac{m-1}{2}} \Gamma\Bigl(\frac{m-1}{2}\Bigr) \Biggl[I_{\frac{m-1}{2}}\bigl(\lambda t \bigr) +I_{\frac{m+1}{2}}\bigl(\lambda t \bigr) \Biggr]
\end{equation}
Hence, for $m = 1$
\begin{equation}
\mathbb{E} \Bigl[ \max_{0\le s\le t} \mathcal{T}(s)\ |\ V(0) = c\Bigr] = e^{-\lambda t}ct \Biggl[I_0\bigl(\lambda t \bigr) +I_1\bigl(\lambda t \bigr) \Biggr]
\end{equation}
\\

For the negatively initially oriented process we have
$$\mathbb{E} \Biggl[ \Bigl(\max_{0\le s\le t} \mathcal{T}(s)\Bigr)^m \ |\ V(0) =- c, N(t) = 2k\Biggr] =$$
\begin{equation}\label{momento-p}
=  \frac{(2k)!(ct)^m}{2^{2k}k!} \Biggl[\frac{\Gamma\Bigl(\frac{m+1}{2}\Bigr)}{\Gamma\Bigl(k+\frac{m+1}{2}\Bigr)} -\frac{\Gamma\Bigl(\frac{m}{2}+1\Bigr)}{\Gamma\Bigl(k+1+\frac{m}{2}\Bigr)}\Biggr]  
\end{equation}
Then, for $m=1$, we have that
\begin{equation}\label{mom1-p}
\mathbb{E} \Biggl[ \max_{0\le s\le t} \mathcal{T}(s) \ |\ V(0) = -c, N(t) = 2k\Biggr] = ct\frac{(2k)!}{k!^22^{2k}}-\frac{ct}{2k+1} =
\end{equation}
$$ =  ct\cdot P\lbrace \max_{0\le s\le t} \mathcal{T}(s) =0\ |\ V(0) =- c, N(t) = 2k \rbrace- \frac{ct}{2k+1}$$
and it is possible to show that there is a maximum in $k=2$, so in the case of four changes of directions.
\\

The moments in the case of an odd number of changes of direction and negatively oriented initial speed can be evaluated from (\ref{momento+d}) and (\ref{momento-p}) by means of the relationship (\ref{sommapesata}).
\hfill $\Diamond$
\end{remark}

\end{document}